\documentclass[12pt]{joa2013}
\usepackage{latexsym, amsthm, amscd, graphicx, euscript, float, subfig}
\pagespan{1}{30}

\usepackage[english]{babel}
\usepackage{amssymb,amsthm,amsmath}

\newtheorem{theorem}{Theorem}[section]
\newtheorem{lemma}[theorem]{Lemma}

\newtheorem{proposition}[theorem]{Proposition}

\theoremstyle{remark}
\newtheorem{definition}[theorem]{Definition}
\newtheorem{remark}[theorem]{Remark}
\newtheorem{example}[theorem]{Example}

\numberwithin{equation}{section}
\thanks{The research of the first author was supported as a part of EUMLS project with grant agreement PIRSES$-$GA$-$2011$-$295164. The research of the second author was supported by a grant received from TUBITAK within 2221-Fellowship Programme for Visiting Scientists and Scientists on Sabbatical Leave}


\begin{document}

\title[Uniform boundedness of pretangent spaces ...] 
{\vspace*{.1cm} Uniform boundedness of pretangent spaces, local constancy of metric derivatives and strong right upper porosity at a point}


\author[Viktoriia Bilet]{\noindent Viktoriia Bilet}
\email{biletvictoriya@mail.ru}
\address{\newline Institute of Applied Mathematics and Mechanics of NASU
\newline R. Luxemburg str. 74, Donetsk 83114,
\newline Ukraine
}

\author[Oleksiy Dovgoshey]{\noindent Oleksiy Dovgoshey}
\email{aleksdov@mail.ru}
\address{\newline Institute of Applied Mathematics and Mechanics of NASU
\newline R. Luxemburg str. 74, Donetsk 83114,
\newline Ukraine
}
\address{\newline Department of Mathematics
\newline Mersin University, Faculty of Art and Sciences,
\newline Mersin 33342,
\newline Turkey
}

\author[Mehmet K\"{u}\c{c}\"{u}kaslan]{\noindent Mehmet K\"{u}\c{c}\"{u}kaslan}
\email{mkucukaslan@mersin.edu.tr}
\address{\newline Department of Mathematics
\newline Mersin University, Faculty of Art and Sciences,
\newline Mersin 33342,
\newline Turkey
}


\begin{abstract}
 Let $(X,d,p)$ be a pointed metric space. A pretangent space to $X$ at $p$ is a metric space consisting of some equivalence classes of convergent to $p$ sequences $(x_n), x_n \in X,$ whose degree of convergence is comparable with a given scaling sequence $(r_n), r_n\downarrow 0.$ A scaling sequence $(r_n)$ is normal if this sequence is eventually decreasing and there is $(x_n)$ such that $\mid d(x_n,p)-r_n\mid=o(r_n)$ for $n\to\infty.$ Let $\mathbf{\Omega_{p}^{X}(n)}$ be the set of  pretangent spaces to $X$ at $p$  with normal scaling sequences.  We prove that $\mathbf{\Omega_{p}^{X}(n)}$ is uniformly bounded if and only if $\{d(x,p): x\in X\}$ is a so-called completely strongly porous set. It is also proved that the uniform boundedness of $\mathbf{\Omega_{p}^{X}(n)}$ is an equivalent of the constancy of metric derivatives of all metrically differentiable mappings on $X$ in the open balls of a fixed radius centered at the marked points of pretangent spaces.
\end{abstract}

\maketitle

\noindent\emph{\small This paper is dedicated to Prof. Matti Vuorinen on the occasion of his 65th birthday.}

\medskip

\noindent\emph{\small The research interests of Prof. Matti Vuorinen are hightly wide. In particular, it includes both the theory of metric spaces and the infinitesimal geometry. The present paper lies in the intersection of these theories. }


\section{Introduction}

Recent achievements in the metric space theory are closely related to some generalizations of the differentiation. A possible but not the only one initial point to develop the theory of a differentiation in metric spaces is the fact that every separable metric space admits an isometric embedding into the dual space of a separable Banach space. It provides a linear structure, and so a differentiation. This approach leads to a rather complete theory of rectifiable sets and currents on metric spaces \cite{ak1, ak2}. The concept of the upper gradient \cite{h, hk, sh}, Cheeger's notion of differentiability for Rademacher's theorem in certain metric measure spaces \cite{ch}, the metric derivative in the studies of metric space valued functions of bounded variation \cite{a1, a2} and the Lipshitz type approach in \cite{ha} are important examples of such generalizations. The generalizations of the differentiability mentioned above give usually nontrivial results only for the assumption that metric spaces have ``sufficiently many'' rectifiable curves.

A new sequential approach to the notion of differentiability for the mapping between the general metric spaces was produced  in \cite{dm1} (see also \cite{dm2}). The main idea of this approach is to introduce a ``metric derivative'' for the mapping between metric spaces by an intrinsic way which does not depend on the possible linear structure in these spaces even if such structures is given. A basic technical tool in \cite{dm1} is pretangent and tangent spaces to an arbitrary metric space $X$ at a point $p.$ The development of the theory of differentiation in metric spaces without linear structure requires the understanding of interrelations between the infinitesimal properties of initial metric space and geometry of pretangent spaces to this initial.

It is almost clear that the boundedness of pretangent spaces to a metric space $(X, d)$ at a point $p\in X$ is closely related to the porosity of the set $S_{p}(X)=\{d(x,p): x\in X\}$ at the point zero. Another not so obvious hypothesis is that the uniform boundedness of all pretangent spaces (to $X$ at $p$) leads to the ``constancy of metric derivatives'' of all ``metrically differentiable'' mappings on $X$ near distinguished points of these pretangent spaces.

What type of porosity of $S_{p}(X)$ at $0$ corresponds to the occasion?

The main result of the paper, Theorem~\ref{th3.10}, gives a complete answer to this question. Furthermore, in Proposition~\ref{Pr5.4}, after a formal definition of ``metric derivatives'', we also obtain a positive answer to the hypothesis formulated above.


Recall some results and terminology related to the pretangent spaces.

Let $(X,d,p)$ be a pointed metric space with a metric $d$ and a marked point $p.$ Fix a
sequence $\tilde{r}$ of positive real numbers $r_n$ tending to zero. In what follows
$\tilde{r}$ will be called a \emph{scaling sequence}. Let us denote by $\tilde{X}$
the set of all sequences of points from X.
\begin{definition}\label{D1.1} Two sequences $\tilde{x}=(x_n)_{n\in \mathbb N}$ and $\tilde{y}=(y_n)_{n\in \mathbb
N},$ $\tilde{x}, \tilde{y} \in \tilde{X}$ are mutually stable with respect to
$\tilde{r}=(r_n)_{n\in \mathbb N}$ if the finite limit
\begin{equation}\label{eq1.2}
\lim_{n\to\infty}\frac{d(x_n,y_n)}{r_n}:=\tilde{d}_{\tilde{r}}(\tilde{x},\tilde{y})=\tilde{d}(\tilde{x},\tilde{y})\end{equation} exists.\end{definition}
We shall say that a family $\tilde{F}\subseteq\tilde{X}$ is \emph{self-stable} (w.r.t.
$\tilde{r}$) if any two $\tilde{x}, \tilde{y} \in \tilde{F}$ are mutually stable. A
family $\tilde{F}\subseteq\tilde{X}$ is \emph{maximal self-stable} if $\tilde{F}$ is
self-stable and for an arbitrary $\tilde{z}\in \tilde{X}$ either $\tilde{z}\in\tilde{F}$
or there is $\tilde{x}\in\tilde{F}$ such that $\tilde{x}$ and $\tilde{z}$ are not
mutually stable.

A standard application of Zorn's lemma leads to the following
\begin{proposition}\label{Pr1.2}Let $(X,d,p)$ be a pointed metric space. Then for every scaling sequence $\tilde{r}=(r_n)_{n\in \mathbb N}$ there exists a maximal self-stable family $\tilde{X}_{p,\tilde{r}}$ such that $\tilde{p}:=(p,p,...)~\in~\tilde{X}_{p,\tilde{r}}.$
\end{proposition}

Note that the condition $\tilde{p}\in\tilde{X}_{p,\tilde{r}}$ implies the equality
$\mathop{\lim}\limits_{n\to\infty}d(x_n,p)=0 $ for every $\tilde{x}=(x_n)_{n\in \mathbb N}\in\tilde{X}_{p,\tilde{r}}.$

Consider a function $\tilde{d}:\tilde{X}_{p,\tilde{r}}\times\tilde{X}_{p,\tilde{r}}\rightarrow\mathbb R$ where
$\tilde{d}(\tilde{x},\tilde{y})=\tilde{d}_{\tilde{r}}(\tilde{x},\tilde{y})$ is defined
by \eqref{eq1.2}. Obviously, $\tilde{d}$ is symmetric and nonnegative. Moreover, the
triangle inequality for $d$ implies
$$\tilde{d}(\tilde{x},\tilde{y})\leq\tilde{d}(\tilde{x},\tilde{z})+\tilde{d}(\tilde{z},\tilde{y})$$
for all $\tilde{x},\tilde{y},\tilde{z}\in\tilde{X}_{p,\tilde{r}}.$ Hence $(\tilde{X}_{p,\tilde{r}},\tilde{d})$
is a pseudometric space.
\begin{definition}\label{D1.3} A pretangent space to the space X (at the point p w.r.t. $\tilde{r}$) is the metric identification of a pseudometric space
$(\tilde{X}_{p,\tilde{r}},\tilde{d}).$\end{definition}

Since the notion of pretangent space is basic for the paper, we remind this metric
identification construction.

Define a relation $\sim$ on $\tilde{X}_{p,\tilde{r}}$ as $\tilde x\sim \tilde y$ if and only if
$\tilde d(\tilde x, \tilde y)=0.$ Then $\sim$ is an equivalence relation. Let us denote
by $\Omega_{p,\tilde r}^{X}$ the set of equivalence classes in $\tilde{X}_{p,\tilde{r}}$ under the
equivalence relation $\sim.$ It follows from general properties of the pseudometric spaces
(see, for example, \cite{k}) that if $\rho$ is defined on $\Omega_{p,\tilde
r}^{X}$ as \begin{equation} \label{eq1.4}\rho(\gamma,\beta):=\tilde d (\tilde x, \tilde
y)\end{equation}for $\tilde x\in \gamma$ and $\tilde y\in \beta,$ then $\rho$ is a
well-defined metric on $\Omega_{p,\tilde r}^{X}.$ By definition, the metric
identification of $(\tilde{X}_{p,\tilde{r}}, \tilde d)$ is the metric space $(\Omega_{p,\tilde
r}^{X}, \rho).$

It should be observed that $\Omega_{p,\tilde r}^{X}\ne \varnothing$ because the constant
sequence $\tilde p$ belongs to $\tilde{X}_{p,\tilde{r}}.$ Thus every pretangent space
$\Omega_{p, \tilde r}^{X}$ is a pointed metric space with the natural distinguished point
$\alpha=\pi (\tilde p),$ (see diagram~\eqref{eq1.5} below).

Let $(n_k)_{k\in\mathbb N}$ be an infinite, strictly increasing sequence of natural
numbers. Let us denote by $\tilde r'$ the subsequence $(r_{n_k})_{k\in \mathbb N}$ of
the scaling sequence $\tilde r=(r_n)_{n\in\mathbb N}$ and let $\tilde
x':=(x_{n_k})_{k\in\mathbb N}$ for every $\tilde x=(x_n)_{n\in\mathbb N}\in\tilde
X.$ It is clear that if $\tilde x$ and $\tilde y$ are mutually stable w.r.t. $\tilde r,$
then $\tilde x'$ and $\tilde y'$ are mutually stable w.r.t. $\tilde r'$ and
$\tilde d_{\tilde r}(\tilde x, \tilde y)=\tilde d_{\tilde r'}(\tilde x', \tilde
y').$ If $\tilde X_{p,\tilde r}$ is a maximal self-stable (w.r.t $\tilde
r$) family, then, by Zorn's Lemma, there exists a maximal self-stable (w.r.t $\tilde
r'$) family $\tilde X_{p,\tilde r'}$ such that $$\{\tilde x':\tilde x \in \tilde
X_{p,\tilde r}\}\subseteq \tilde X_{p,\tilde r'}.$$ Denote by $in_{\tilde r'}$ the map
from $\tilde X_{p,\tilde r}$ to $\tilde X_{p,\tilde r'}$ with $in_{\tilde r'}(\tilde
x)=\tilde x'$ for all $\tilde x\in\tilde X_{p,\tilde r}.$ It follows from \eqref{eq1.4}
that after metric identifications $in_{\tilde r'}$ passes to an isometric embedding
$em':\Omega_{p,\tilde r}^{X}~\rightarrow~\Omega_{p,\tilde r'}^{X}$ under which the
diagram
\begin{equation} \label{eq1.5}
\begin{array}{ccc}
\tilde X_{p, \tilde r} & \xrightarrow{\ \ \mbox{\emph{in}}_{\tilde r'}\ \ } &
\tilde X_{p, \tilde r^{\prime}} \\
\!\! \!\! \!\! \!\! \! \pi\Bigg\downarrow &  & \! \!\Bigg\downarrow \pi^{\prime}
\\
\Omega_{p, \tilde r}^{X} & \xrightarrow{\ \ \mbox{\emph{em}}'\ \ \ } & \Omega_{p, \tilde
r^{\prime}}^{X}
\end{array}
\end{equation}is commutative. Here $\pi$ and
$\pi'$ are the natural projections, $$\pi(\tilde x):=\{\tilde y \in \tilde X_{p,\tilde
r}: \tilde d_{\tilde r}(\tilde x, \tilde y)=0\} \quad \mbox{and} \quad \pi'(\tilde x):=\{\tilde y \in
\tilde X_{p,\tilde r'}: \tilde d_{\tilde r'}(\tilde x, \tilde y)=0\}.$$

Let $X$ and $Y$ be metric spaces. Recall that a map $f:X\rightarrow Y$ is called an
\emph{isometry} if $f$ is distance-preserving and onto.

\begin{definition}\label{D1.4}A pretangent $\Omega_{p,\tilde
r}^{X}$ is tangent if $em':\Omega_{p,\tilde r}^{X}\rightarrow \Omega_{p,\tilde r'}^{X}$
is an isometry for every~$\tilde r'.$\end{definition}

The following lemmas will be used in sections 3 and 4 of the paper.

\begin{lemma}\label{Lem1.6}\emph{\cite{dak}}
Let $\mathbf{\mathfrak{B}}$ be a countable subfamily
of $\tilde X$ and let $\tilde\rho=(\rho_n)_{n\in\mathbb N}$ be a scaling sequence.
Suppose that the inequality $$\limsup_{n\to\infty}\frac{d(b_n,p)}{\rho_n}<\infty$$ holds for every $\tilde b=(b_n)_{n\in\mathbb N}\in\mathbf{\mathfrak{B}}.$ Then there is an infinite
subsequence $\tilde \rho'$ of $\tilde \rho$ such that the family
$\mathbf{\mathfrak{B'}}=\{\tilde b': \tilde b \in
\mathbf{\mathfrak{B'}}\}$ is self-stable w.r.t. $\tilde \rho'.$
\end{lemma}
The next lemma follows from Corollary 3.3 of \cite{dov}.

\begin{lemma}\label{Lem1.7}
Let $(X,d)$ be a metric space and let $Y, Z$ be dense subsets of $X.$ Then for every $p\in Y \cap Z$ and every $\Omega_{p, \tilde r}^{Y}$ there are $\Omega_{p, \tilde r}^{Z}$ and an isometry $f:\Omega_{p, \tilde r}^{Y}\rightarrow \Omega_{p, \tilde r}^{Z}$ such that $f(\alpha_Y)=\alpha_Z$ where $\alpha_Y$ and $\alpha_Z$ are the marked points of $\Omega_{p, \tilde r}^{Y}$ and $\Omega_{p, \tilde r}^{Z}$ respectively.
\end{lemma}

\section{Completely strongly porous sets}

The notion of ``porosity'' for the first time appeared in some early works of Denjoy \cite{d1},  \cite{d2} and Khintchine \cite{chi} and then arose independently in the study of cluster sets in 1967 by Dol\v{z}enko \cite{dol}. A useful collection of facts related to the notion of porosity can be found, for example, in \cite{fh}, \cite{hv}, \cite{t} and \cite{tk}. The porosity appears naturally in many problems and plays an implicit role in various regions of analysis (e. g., the cluster sets \cite{z1}, the Julia sets \cite{pr}, the quasisymmetric maps \cite{va},  the differential theory \cite{kps}, the theory of generalized subharmonic functions \cite{dr} and so on). The reader can also consult \cite{z2} and \cite{z3} for more information.

Let us recall the definition of the right upper porosity. Let $E$ be a subset of $\mathbb R^{+}=[0,\infty).$
\begin{definition}\emph{\cite{t}}\label{D1}
The right upper porosity of $E$ at 0 is the quantity
\begin{equation}\label{L1}
p^{+}(E,0):=\limsup_{h\to 0^{+}}\frac{\lambda(E,0,h)}{h}
\end{equation}
where $\lambda(E,0,h)$ is the length of the largest open subinterval of $(0,h),$ which could be the empty set $\varnothing,$ that
contains no points of $E$. The set $E$ is strongly porous at 0 if $p^{+}(E,0)=1.$
\end{definition}
Let $\tilde \tau=(\tau_n)_{n\in\mathbb N}$ be a sequence of real numbers. We shall say
that $\tilde \tau$ is eventually decreasing if the inequality $\tau_{n+1}\le\tau_{n}$ holds
for sufficiently large $n.$ Write $\tilde E_{0}^{d}$ for the set of eventually decreasing
sequences $\tilde \tau$ with $\mathop{\lim}\limits_{n\to\infty}\tau_{n}=0$ and
$\tau_{n}\in E\setminus \{0\}$ for $n\in\mathbb N.$

Define $\tilde I_{E}^{d}$ to be the set of sequences of open intervals $(a_n,b_n)\subseteq
\mathbb R^{+}, n\in\mathbb N,$ meeting the conditions:

\bigskip
$\bullet$ \emph{Each $(a_n, b_n)$ is a connected component of the set $Ext E=Int(\mathbb
R^{+}\setminus E),$ i.e., $(a_n,b_n)\cap E=\varnothing$ but  $$((a,b)\ne (a_n, b_n))\Rightarrow((a,b)\cap E \ne
\varnothing)$$} \emph{for every}
$(a,b)\supseteq(a_n,b_n);$

$\bullet$ $(a_n)_{n\in\mathbb N}$ \emph{is eventually decreasing};

\emph{$\bullet$
$\mathop{\lim}\limits_{n\to\infty}a_{n}=0$ and
$\mathop{\lim}\limits_{n\to\infty}\frac{b_n-a_n}{b_n}=1.$}

\bigskip

Define also an equivalence $\asymp$ on the set of sequences of strictly positive
numbers as follows. Let $\tilde a=(a_n)_{n\in\mathbb N}$ and
$\tilde{\gamma}=(\gamma_n)_{n\in\mathbb N}.$ Then $\tilde a \asymp \tilde {\gamma}$ if
there are some constants $c_1, c_2
>0$ such that
$
c_1 a_n < \gamma_n < c_2 a_n
$
for $n\in\mathbb N.$

The next definition is an equivalent  form of Definition 1.4 from \cite{bd2}.
\begin{definition}\label{D2*}
Let $E\subseteq\mathbb R^{+}$ and let $\tilde \tau
\in \tilde E_{0}^{d}.$ The set $E$ is $\tilde \tau$-strongly porous at 0 if there is a
sequence $\{(a_n, b_n)\}_{n\in\mathbb N}\in\tilde I_{E}^{d}$ such that
$
\tilde\tau \asymp \tilde a
$
where $\tilde a=(a_n)_{n\in\mathbb N}.$ The set $E$ is completely strongly porous
at 0 if $E$ is $\tilde \tau$-strongly porous at 0 for every $\tilde \tau \in \tilde
E_{0}^{d}.$
\end{definition}

We denote by \textbf{\emph{CSP}} the set of all completely strongly porous at 0 subsets of $\mathbb R^{+}$. It is clear that every $E\in\textbf{\emph{CSP}}$ is strongly porous at 0 but not conversely. Moreover, if 0 is an isolated point of $E\subseteq\mathbb R^{+},$ then $E\in \textbf{\emph{CSP}}.$

The next lemma immediately follows from Definition~\ref{D2*}.

\begin{lemma}\label{Lem2.3}
Let $E\subseteq\mathbb R^{+},$
$\tilde\gamma\in\tilde E_{0}^{d},$  $\{(a_n, b_n)\}_{n\in\mathbb N}\in\tilde I_{E}^{d}$
and let $\tilde a=(a_n)_{n\in\mathbb N}.$ The equivalence
$\tilde\gamma\asymp\tilde a$ holds if and only if we have
\begin{equation*}\label{L2*}
\limsup_{n\to\infty}\frac{a_n}{\gamma_n}<\infty \quad\mbox{and}\quad \gamma_n \le a_n
\end{equation*} for sufficiently large $n.$
\end{lemma}

Define a set $\mathbb N_{N_1}$ as $\{N_1, N_1+1,...\}$ for $N_1\in\mathbb N.$

\begin{definition}\label{univ} Let $$\tilde A =\{(a_n,b_n)\}_{n\in\mathbb
N}~\in~\tilde I_{E}^{d}\quad\mbox{and}\quad \tilde L =\{(l_n,m_n)\}_{n\in\mathbb
N}~\in~\tilde I_{E}^{d}.$$ Write $\tilde A\preceq\tilde L$ if there are $N_1 \in\mathbb N$ and $f:~\mathbb N_{N_1}~\rightarrow~\mathbb N$ such that
$ a_n = l_{f(n)}$ for every $n\in\mathbb N_{N_1}.$
$\tilde L$ is universal if $\tilde B \preceq \tilde
L$ holds for every $\tilde B\in\tilde I_{E}^{d}.$
\end{definition}

Let $\tilde L=\{(l_n, m_n)\}_{n\in\mathbb N}\in\tilde I_{E}^{d}$ be universal and let
\begin{equation}\label{L13}
M(\tilde L):=\limsup_{n\to\infty}\frac{l_n}{m_{n+1}}.
\end{equation}

In what follows $ac E$ means the set of all accumulation points of a set $E.$
\begin{theorem}\emph{\cite{bd2}}\label{ImpTh}
Let $E\subseteq\mathbb R^{+}$ be strongly porous at 0 and $0\in ac E.$ Then $E\in \textbf{CSP}$ if and only if there is an universal $\tilde L\in \tilde I_{E}^{d}$ such that $M(\tilde L)<\infty.$
\end{theorem}

Note that the quantity $M(\tilde L)$ depends from the set $E$ only (for details see \cite{bd2}).
The following lemma is used in next part of the paper.

\begin{lemma}\label{Pr5}\emph{\cite{bd2}}
Let $E\subseteq\mathbb R^{+}$ and let
$\tilde \tau=(\tau_n)_{n\in\mathbb N}\in\tilde E_{0}^{d}.$ Then $E$ is $\tilde\tau$-strongly porous at 0 if and only if there is a constant $k\in (1, \infty)$ such that for every $K\in (k, \infty)$ there exists $N_{1}(K)\in \mathbb N$ with
$ (k\tau_n, K\tau_n)\cap E = \varnothing$ for every $n
\ge N_1(K).$
\end{lemma}

Let us consider now a simple set belonging to \emph{\textbf{CSP}}.

\begin{example}\label{ex4.4.7}
Let $(x_n)_{n\in\mathbb N}$ be strictly decreasing sequence of positive real numbers with $\mathop{\lim}\limits_{n\to\infty}\frac{x_{n+1}}{x_n}=0.$ Define a set $W$ as $$W=\{x_n: n\in\mathbb N\},$$ i.e., $W$ is the range of the sequence $(x_n)_{n\in\mathbb N}.$ Then $W\in \textbf{\emph{CSP}}$ and $\tilde L=\{(x_{n+1}, x_n)\}_{n\in\mathbb N}\in\tilde I_{W}^{d}$ is universal with $M(\tilde L)=1.$
\end{example}

\begin{proposition}\label{pr4.4.8}
 Let $E\subseteq\mathbb R^{+}.$ Then the inclusion
\begin{equation}\label{4.4.22}
\{E\cup A: A\in \textbf{CSP}\}\subseteq \textbf{CSP}
\end{equation}
holds if and only if $0$ is an isolated point of $E.$
\end{proposition}

\begin{remark}\label{rem4.4.9}
Inclusion \eqref{4.4.22} means that $E\cup A\in \textbf{\emph{CSP}}$ for every $A\in \textbf{\emph{CSP}}.$
\end{remark}

\noindent {\bf Proof of Proposition~\ref{pr4.4.8}.} If $0\notin ac E,$ then \eqref{4.4.22} follows almost directly and we omit the details here. Suppose $0\in ac E.$  Then there is a sequence $(\tau_n)_{n\in\mathbb N}$ such that $\tau_n\in E$ and $\tau_{n+1}\le 2^{-n^{2}}\tau_n$ for every $n\in\mathbb N.$ Let $M_{1}, M_2,..., M_{k},...$ be an infinite partition of $\mathbb N,$ $$\mathop{\cup}\limits_{k=1}^{\infty}M_{k}=\mathbb N, \, M_{i}\cap M_{j}=\varnothing \,\mbox{if} \, i\ne j$$ such that card$M_{k}=$card$\mathbb N$ for every $k$ and $\nu(1)<\nu(2)<...<\nu(k)...$ where \begin{equation}\label{4.4.23}\nu(k)=\mathop{\min}\limits_{n\in M_k}n.\end{equation} Let $n\in\mathbb N$ and let $m(n)$ be the index such that $n\in M_{m(n)}.$ For every $n\in\mathbb N$ define $\tau_n^{*}$ as $2^{-m(n)}\tau_n.$ Write
\begin{equation*}
E_{1}=\{\tau_n: n\in\mathbb N\} \quad\mbox{and}\quad E_{1}^{*}=\{\tau_{n}^{*}: n\in\mathbb N\}.
\end{equation*}
Using Lemma~\ref{Pr5} we can show that $E_{1}\cup E_{1}^{*}$ is not $\tau^{*}$-strongly porous with $\tilde\tau^{*}=(\tau_n)_{n\in\mathbb N}.$ Consequently $E_{1}\cup E_{1}^{*} \notin \textbf{\emph{CSP}}.$ It implies that $E \cup E_{1}^{*} \notin \textbf{\emph{CSP}}$ because $E_{1}\subseteq E.$ To complete the proof, it suffices to show that $E_{1}^{*}\in \emph{\textbf{CSP}}.$ To this end, we note that \eqref{4.4.23} and the inequalities $\nu(1)<\nu(2)<...<\nu(k)...$ imply that $m(n)\le n$ for every $n\in\mathbb N.$ Indeed, if $m(n)=k,$ then we have $$n\ge\nu(k)=(\nu(k)-\nu(k-1))+(\nu(k-1)-\nu(k-2))+...+ (\nu(2)-\nu(1))+\nu(1)$$ $$\ge (k-1)+\nu(1)=k=m(n).$$ Consequently $$\tau_{n}^{*}=2^{-m(n)}\tau_n\ge 2^{-n}\tau_n\ge2^{-n^{2}}\tau_n \ge\tau_{n+1}\ge\tau_{n+1}^{*}$$ for every $n\in\mathbb N.$ It follows from that $$\lim_{n\to\infty}\frac{\tau_{n}^{*}}{\tau_{n+1}^{*}}\ge\lim_{n\to\infty}\frac{2^{-n}\tau_n}{2^{-n^{2}}\tau_n }=\lim_{n\to\infty}2^{n^{2}-n}=+\infty.$$
Thus, as in Example~\ref{ex4.4.7}, we have $E_{1}^{*}\in \textbf{\emph{CSP}}.$ $\qquad\qquad\qquad\qquad\qquad\qquad\blacksquare$


\section{Uniform boundedness and \textbf{\emph{CSP}}}

Let $\mathfrak F=\{(X_{i}, d_{i}): i \in I\}$ be a nonempty family
of  metric spaces. The family $\mathfrak F$ is \emph{uniformly bounded} if there is a constant
$c>0$ such that the inequality
$
\textrm{diam}X_i <c
$
holds for every $i\in I.$ If all metric spaces $(X_i,d_i)\in\mathfrak
{F}$ are pointed with marked points $p_i\in X_i,$ then the uniform boundedness of
$\mathfrak {F}$ can be described by the next way. Define
\begin{equation}\label{L4.1}
\rho^{*}(X_i):=\sup_{x\in X_i}d_{i}(x,p_i) \quad \mbox{and} \quad R^{*}(\mathfrak {F}):=\sup_{i\in I}\rho^{*}(X_i).
\end{equation}
The family $ \mathfrak {F}$ is uniformly bounded if and only if
$
R^{*}(\mathfrak {F})<\infty.
$

\begin{proposition}\label{Pr4.1}
Let $(X,d,p)$ be a pointed metric space and let $\mathbf{\Omega_{p}^{X}}$ be the set of all pretangent spaces to $X$ at $p.$ The following statements are equivalent.
\newline $\mathrm{(i_1)}$ \textit{The family $\mathbf{\Omega_{p}^{X}}$ is uniformly bounded.}
\newline $\mathrm{(i_2)}$ \textit{The point $p$ is an isolated point of $X$.}
\end{proposition}
\begin{proof}The implication $\mathrm{(i_2)} \Rightarrow \mathrm{(i_1)}$ follows directly from the definitions. To prove $\mathrm{(i_1)} \Rightarrow \mathrm{(i_2)}$ suppose that $p\in ac X.$  Let $\tilde x=(x_n)_{n\in\mathbb N}\in\tilde X$ be a sequence of
distinct points of X such that $\mathop{\lim}\limits_{n\to\infty}d(x_{n},p)~=~0.$ For
$t>0$ define the scaling sequence $\tilde r_{t}=(r_{n,t})_{n\in\mathbb N}$ with
$r_{n,t}=\frac{d(x_n,p)}{t}.$ It follows at once from Definition 1.1 that $\tilde x$ and
$\tilde p$ are mutually stable w.r.t $\tilde r_{t}$ and
\begin{equation}\label{L4.4}
\tilde d_{\tilde r_{t}}(\tilde x, \tilde p)=t.
\end{equation}
Let $\tilde X_{p,\tilde r_{t}}$ be a maximal self-stable family meeting the relation
$\tilde x\in\tilde X_{p,\tilde r_{t}}.$ Equality \eqref{L4.4} implies the inequality
$$\textrm{diam} \, \Omega_{p,\tilde r_{t}}^{X}\ge t,$$ where $\Omega_{p,\tilde r_{t}}^{X}=\pi(\tilde X_{p,\tilde
r_{t}}).$ Consequently the family $\mathbf{\Omega_{p}^{X}}$ is not uniformly bounded.
The implication $\mathrm{(i_1)} \Rightarrow \mathrm{(i_2)}$ follows.
\end{proof}

The proposition above shows that the question on the uniform boundedness can be
informative only for some special subfamilies of $\mathbf{\Omega_{p}^{X}}.$ We can
narrow down the family $\mathbf{\Omega_{p}^{X}}$ by the way of consideration some special
scaling sequences.

\begin{definition}\label{D4.2}
Let $(X,d,p)$ be a pointed metric space and let $p\in ac X.$ A
scaling sequence $(r_{n})_{n\in\mathbb N}$ is normal if $(r_n)_{n\in\mathbb N}$ is eventually decreasing and there is $
(x_{n})_{n\in\mathbb N}\in\tilde X$ such that
\begin{equation}\label{L4.4*}
\lim_{n\to\infty}\frac{d(x_n,p)}{r_{n}}=1.
\end{equation}
\end{definition}

\begin{proposition}\label{prop}The following statements hold for every pointed metric space $(X,d,p).$
\newline $(i_1)$ If $\Omega_{p,\tilde r}^{X}$ contains at least two distinct points, then there are $c>0$ and a subsequence $(r_{n_k})_{k\in\mathbb N}$ of $\tilde r$ so that the sequence $(cr_{n_k})_{k\in\mathbb N}$ is normal.
\newline $(i_2)$ If $(x_n)_{n\in\mathbb N}\in\tilde X$ and \eqref{L4.4*} holds, then there is an infinite increasing sequence $(n_k)_{k\in\mathbb N}$ so that $(d(x_{n_k}, p))_{k\in\mathbb N}$ is decreasing.
\newline $(i_3)$ If $\tilde r$ is a normal scaling sequence,then there is $(x_n)_{n\in\mathbb N}\in\tilde X$ such that \eqref{L4.4*} holds and $(d(x_n, p))_{n\in\mathbb N}$ is eventually decreasing.
\end{proposition}
\begin{proof}
It is easily verified that $(i_1)$ and $(i_2)$ hold. To verify $(i_3)$ observe that there is $(y_n)_{n\in\mathbb N}\in\tilde X$ which satisfies $d(y_n, p)>0$ for every $n\in\mathbb N$ and \eqref{L4.4*} with $(x_n)_{n\in\mathbb N}=(y_n)_{n\in\mathbb N}.$ Let $m(n)\in\mathbb N$ meet the conditions $m(n)\le n$ and $d(y_{m(n},p)=\mathop{\min}\limits_{1\le i \le n}d(y_i, p).$ The conditions $\mathop{\lim}\limits_{n\to\infty}y_n=p$ and $d(y_n, p)>0$ for $n\in\mathbb N$ imply that $m(n)\to\infty$ as $n\to\infty.$ Since $\tilde r$ is eventually decreasing, there is $n_{0}\in\mathbb N$ such that $(r_{m(n)})=r_{n}$ for every $n\ge n_{0}.$ Consequently, we obtain $$1=\lim_{n\to\infty}\frac{d(y_{m(n)},p)}{r_{m(n)}}=\lim_{n\to\infty}\frac{d(y_{m(n)},p)}{r_{n}}\le\lim_{n\to\infty}\frac{d(y_{n},p)}{r_{n}}=1.$$ It is clear that $(d(y_{m(n)}, p))_{n\in\mathbb N}$ is decreasing. Thus $(i_3)$ holds with $(x_n)_{n\in\mathbb N}=(y_{m(n)})_{n\in\mathbb N}.$
\end{proof}

Write $\mathbf{\Omega_{p}^{X}(n)}$ for the
set of pretangent spaces $\Omega_{p,\tilde r}^{X}$ with normal scaling sequences.
Under what conditions the family $\mathbf{\Omega_{p}^{X}(n)}$ is uniformly bounded?

\begin{remark}\label{Rem4.3}
Of course, the property of scaling sequence $\tilde r$ to be normal depends on the
underlaying space $(X,d,p)$. Nevertheless for every pointed metric space $(X,d,p)$ a
scaling sequence $\tilde r$ is normal for this space if and only if it
is normal for the space $(S_{p}(X), |\cdot |, 0)$. We shall use this simple fact
below in Proposition~\ref{Pr4.4}.
\end{remark}

In the next proposition we define $\mathbf{\Omega_{0}^{E}(n)}$ to be the set of all pretangent spaces to the
distance set $E=S_p (X)$ at 0 w.r.t. normal scaling sequences.
\begin{proposition}\label{Pr4.4} Let $(X,d,p)$ be a pointed metric space and let $E=S_p (X).$
Then we have
\begin{equation}\label{L4.5}
R^{*}(\mathbf{\Omega_{0}^{E}(n)})=R^{*}(\mathbf{\Omega_{p}^{X}(n)})
\end{equation}
where
$R^{*}(\mathbf{\Omega_{p}^{X}(n)})$ and $R^{*}(\mathbf{\Omega_{0}^{E}(n)})$ are defined by \eqref{L4.1}
with $\mathfrak {F}=\mathbf{\Omega_{p}^{X}(n)}$ and $\mathfrak {F}=\mathbf{\Omega_{0}^{E}(n)}$ respectively.
\end{proposition}
\begin{proof}
If $p\notin ac X$, then the set of normal scaling sequences is
empty. Consequently we have $\mathbf{\Omega_{0}^{E}(n)}=\mathbf{\Omega_{p}^{X}(n)}=\varnothing,$ so we suppose that $p\in ac X.$

 For each normal scaling sequence $\tilde r$ and every $\tilde x \in
\tilde X$ having the finite limit $\mathop{\lim}\limits_{n\to\infty}\frac{d(x_n,p)}{r_n}$ we
can find $(s_n)_{n\in\mathbb N}\in\tilde E$ such that
\begin{equation}\label{**}\lim_{n\to\infty}\frac{d(x_n,p)}{r_n}=\lim_{n\to\infty}\frac{s_n}{r_n}.\end{equation}
Hence the inequality
\begin{equation}\label{L4.7}
\rho(\alpha, \beta)\le R^{*}(\mathbf{\Omega_{0}^{E}(n)})
\end{equation}
holds with $\alpha=\pi(\tilde p)$ for every $\beta\in\Omega_{p,\tilde r}^{X}$ and every
$\Omega_{p,\tilde r}^{X}\in\mathbf{\Omega_{p}^{X}(n)}.$ Taking supremum over all $\Omega_{p,\tilde r}^{X}\in\mathbf{\Omega_{p}^{X}(n)}$ and $\beta\in\Omega_{p,\tilde r}^{X},$ we get
\begin{equation}\label{L4.6} R^{*}(\mathbf{\Omega_{0}^{E}(n)})\ge R^{*}(\mathbf{\Omega_{p}^{X}(n)}).\end{equation} It still remains to prove the inequality
\begin{equation}\label{L4.8}
 R^{*}(\mathbf{\Omega_{0}^{E}(n)})\le R^{*}(\mathbf{\Omega_{p}^{X}(n)}).
\end{equation}
As is easily seen, for every normal scaling $\tilde r$ and every $\tilde s\in
\tilde E$ with $\mathop{\lim}\limits_{n\to\infty}\frac{s_n}{r_n}<\infty,$ there is $\tilde x\in\tilde X$ satisfying \eqref{**}. Now reasoning as in the proof of \eqref{L4.6} we obtain \eqref{L4.8}.
Equality \eqref{L4.5} follows from \eqref{L4.6} and
\eqref{L4.8}.
\end{proof}

\begin{lemma}\label{lem4.5}
Let $E\subseteq\mathbb R^{+}$ and let $0\in ac E.$ If the inequality
\begin{equation}\label{L4.9}
R^{*}(\mathbf{\Omega_{0}^{E}(n)})<\infty
\end{equation} holds, then $E\in \textbf{CSP}$.
\end{lemma}
\begin{proof}
Suppose that \eqref{L4.9} holds but there is $\tilde \tau=(\tau_n)_{n\in\mathbb
N}\in\tilde E_0^{d}$ such that $E$ is not $\tilde\tau$-strongly porous at 0. Then, by
Lemma~\ref{Pr5}, for every $k>1$ there is $K\in(k,\infty)$ such that
$
(k\tau_n, K\tau_n)\cap E\ne\varnothing
$ for all $n$ belonging to an infinite set $A\subseteq\mathbb N.$ Let us put \begin{equation}\label{L4.11}
k=2R^{*}(\mathbf{\Omega_{0}^{E}(n)}).
\end{equation}It simply follows from \eqref{L4.1} and Definition~\ref{D4.2} that $R^{*}(\mathbf{\Omega_{0}^{E}(n)})\ge
1.$ Thus $k\ge 2.$ Consequently we can find $K\in(k, \infty)$ and an infinite set
$A=\{n_1,...,n_j, ...\}\subseteq\mathbb N,$ such that for every $n_j\in A$ there
is $x_j \in E$ satisfying the double inequality
\begin{equation}\label{L4.12}
k<\frac{x_j}{\tau_{n_j}}<K.
\end{equation}Thus the sequence $\left(\frac{x_j}{\tau_{n_j}}\right)_{j\in\mathbb N}$ is bounded.
Hence it contains a convergent subsequence.
Passing to this subsequence we obtain
\begin{equation}\label{L4.13}
\lim_{j\to\infty}\frac{x_j}{\tau_{n_j}}<\infty.
\end{equation}
Now \eqref{L4.11} and \eqref{L4.12} imply
\begin{equation}\label{L4.14}
\lim_{j\to\infty}\frac{x_j}{\tau_{n_j}}=\lim_{j\to\infty}\frac{|0-x_j|}{\tau_{n_j}}\ge
2R^{*}(\mathbf{\Omega_{0}^{E}(n)}).
\end{equation}
The scaling sequence $\tilde r=(r_j)_{j\in\mathbb N}$ with $r_j=\tau_{n_j},$
$j\in\mathbb N,$ is normal. The existence of
finite limit \eqref{L4.13} implies that $\tilde x=(x_j)_{j\in\mathbb N}$ and $\tilde
0$ are mutually stable w.r.t $\tilde r.$ Consequently there is a maximal self-stable
family $\tilde E_{0, \tilde r}$ such that $\tilde x, \tilde 0\in\tilde E_{0, \tilde r}.$ Write $\Omega_{0, \tilde
r}^{E}$ for the metric identification of $\tilde E_{0, \tilde r}$ and $\alpha$ for the
natural projection of $\tilde 0.$ Using \eqref{L4.14} and \eqref{L4.1}, we obtain
$$R^{*}(\mathbf{\Omega_{0}^{E}(n)})\ge\sup_{\gamma\in\Omega_{0,\tilde r}^{E}}\rho(\alpha, \gamma)\ge
2R^{*}(\mathbf{\Omega_{0}^{E}(n)}).$$ The last double inequality is inconsistent because
$1\le R^{*}(\mathbf{\Omega_{0}^{E}(n)})< \infty.$ Thus if \eqref{L4.9} holds, then $E$ is
$\tilde \tau$-strongly porous at 0, as required.
\end{proof}

Let $\tilde \tau\in\tilde E_{0}^{d}.$ Define a subset $\tilde I_{E}^{d} (\tilde \tau)$ of
the set $\tilde I_E^{d}$ by the rule:
$$(\{(a_n ,b_n)\}_{n\in\mathbb N}\in\tilde I_{E}^{d}(\tilde \tau))\Leftrightarrow (\{(a_n ,b_n)\}_{n\in\mathbb N}\in\tilde I_{E}^{d}\, \mbox{and}\, $$ $$\tau_n\le a_n \, \mbox{for sufficiently large}\, n\in\mathbb N).$$
Write
\begin{equation}\label{L9}
C(\tilde \tau):=\inf(\limsup_{n\to\infty}\frac{a_n}{\tau_n})\quad \mbox{and}\quad
C_E:=\sup_{\tilde \tau\in\tilde E_{0}^{d}}C(\tilde\tau)
\end{equation} where the infimum is taken over all $\{(a_n ,b_n)\}_{n\in\mathbb N}\in\tilde I_{E}^{d}(\tilde \tau).$

Using Theorem~\ref{ImpTh} we can prove the following

\begin{proposition}\label{P2}\emph{\cite{bd2}}
Let $E\subseteq\mathbb R^{+}$ and $\tilde\tau\in \tilde E_{0}^{d}.$ The set $E$ is $\tilde\tau$-strongly porous at $0$ if and only if $C(\tilde\tau)<\infty.$  The membership $E\in \textbf{CSP}$ holds if and only if $C_E <\infty.$
\end{proposition}

\begin{remark}\label{R1}
If $E\subseteq\mathbb R^{+}, \, i=1,2, \,  \{(a_{n}^{(i)}, b_{n}^{(i)})\}_{n\in\mathbb N}\in\tilde I_{E}^{d}$ and $\tilde a^{1}\asymp \tilde a^{2}$ where $\tilde a^{i}=(a_{n}^{(i)})_{n\in\mathbb N}, \, $ then there is $n_0 \in\mathbb N$ such that $(a_{n}^{(1)}, b_{n}^{(1)})=(a_{n}^{(2)}, b_{n}^{(2)})$ for every $n\ge n_0.$ Consequently if $E$ is $\tilde\tau$-strongly porous and $\{(a_n, b_n)\}_{n\in\mathbb N}\in\tilde I_{E}^{d}(\tilde\tau),$ then we have
\begin{equation*}\label{z1}
\mbox{either}\quad
\limsup_{n\to\infty}\frac{a_n}{\tau_n}=\infty
\quad\mbox{or}\quad \limsup_{n\to\infty}\frac{a_n}{\tau_n}=C(\tilde\tau)<\infty.\end{equation*}
\end{remark}

\begin{lemma}\emph{\cite{bd2}}\label{Lem2.15}
Let $E\in \textbf{CSP}.$ If $\tilde L\in\tilde
I_{E}^{d}$ is universal,
then
$
M(\tilde L)=C_E
$
where the quantities $M(\tilde L)$ and $C_E$ are defined by \eqref{L13} and \eqref{L9}
respectively.
\end{lemma}

\begin{proposition}\label{Pr4.6}
Let $E\subseteq\mathbb R^{+}$ and let $0\in ac E.$ Then the
equality
\begin{equation}\label{L4.15}
C_E =R^{*}(\mathbf{\Omega_{0}^{E}(n)})
\end{equation} holds.
\end{proposition}
\begin{proof}
Let us prove the inequality
\begin{equation}\label{L4.16}
C_E \ge R^{*}(\mathbf{\Omega_{0}^{E}(n)}).
\end{equation}This is trivial if $C_E=\infty.$  Suppose
that
$
C_E <\infty.
$
Inequality \eqref{L4.16} holds if, for every normal scaling sequence $\tilde
r=(r_n)_{n\in\mathbb N}$ and each $\tilde y=(y_n)_{n\in\mathbb N}\in\tilde E_{0}^{d},$
the existence of the finite limit
$\mathop{\lim}\limits_{n\to\infty}\frac{y_n}{r_n}$ implies the
inequality
\begin{equation}\label{L4.18}
\lim_{n\to\infty}\frac{y_n}{r_n}\le C_E.
\end{equation}
Since $\tilde r$ is normal, Proposition~\ref{prop} implies that there is $\tilde x=(x_n)_{n\in\mathbb N}\in\tilde E_{0}^{d}$
with $\mathop{\lim}\limits_{n\to\infty}\frac{r_n}{x_n}=1.$ Consequently \eqref{L4.18} holds if
and only if
\begin{equation}\label{L4.19}
\lim_{n\to\infty}\frac{y_n}{x_n}\le C_E.
\end{equation}
If $\mathop{\lim}\limits_{n\to\infty}\frac{y_n}{x_n}=0,$ then \eqref{L4.19} is trivial.
Suppose that $0<\mathop{\lim}\limits_{n\to\infty}\frac{y_n}{x_n}<\infty.$ The last double
inequality implies the equivalence $\tilde x\asymp\tilde y.$ In accordance with
Proposition~\ref{P2}, $E\in \textbf{\emph{CSP}}$ if and only if $C_{E}<\infty$
holds. Hence $E$ is $\tilde x$-strongly porous at 0. Consequently there is $\{(a_n,
b_n)\}_{n\in\mathbb N}\in\tilde I_{E}^{d}$ such that $\tilde x\asymp\tilde a.$ The relations
$\tilde x\asymp\tilde y$ and $\tilde x\asymp\tilde a$ imply $\tilde y\asymp\tilde a.$
Using Lemma~\ref{Lem2.3} we can find $N_0 \in \mathbb N$ such that $y_n \le a_n$ for
$n\ge N_0.$ Consequently we have $\frac{y_n}{x_n}\le\frac{a_n}{x_n}$ for $n\ge N_0,$
 which implies
$$\lim_{n\to\infty}\frac{y_n}{x_n}\le\limsup_{n\to\infty}\frac{a_n}{x_n}\le C(\tilde\tau)\le
C_E$$ (see \eqref{L9}). Inequality \eqref{L4.16} follows.

To prove \eqref{L4.15}, it
still remains to verify the inequality
\begin{equation}\label{L4.20}
C_E\le R^{*}(\mathbf{\Omega_{0}^{E}(n)}).
\end{equation}
It is trivial if $R^{*}(\mathbf{\Omega_{0}^{E}(n)})=\infty.$ Suppose that
\begin{equation}\label{L4.21}
 R^{*}(\mathbf{\Omega_{0}^{E}(n)})<\infty.
\end{equation}
Inequality \eqref{L4.20}
holds if
\begin{equation}\label{L4.22}
 C(\tilde x)\le R^{*}(\mathbf{\Omega_{0}^{E}(n)}).
\end{equation} for every $\tilde x\in\tilde E_{0}^{d}.$ Let $E^{1}$ denote the closure of the set $E$ in $\mathbb R^{+}$ and let $\tilde x\in\tilde E_{0}^{d}.$ It follows at once from Lemma~\ref{Lem1.7}
that
$R^{*}(\mathbf{\Omega_{0}^{E}(n)})=R^{*}(\mathbf{\Omega_{0}^{E^{1}}(n)}).$ Consequently
\eqref{L4.22} holds if $C(\tilde x)\le R^{*}(\mathbf{\Omega_{0}^{E^{1}}(n)}).$ By
Lemma~\ref{lem4.5}, inequality \eqref{L4.21} implies that $E\in \textbf{\emph{CSP}}$. Hence, by Lemma~\ref{Lem2.3}, there is $\{(a_n, b_n)\}_{n\in\mathbb
N}\in\tilde I_{E}^{d}$ such that \begin{equation}\label{L4.23}
 \limsup_{n\to\infty}\frac{a_n}{x_n}<\infty
\end{equation} and $a_n \ge x_n$ for sufficiently large $n.$ Inequality \eqref{L4.23}
implies the equality
\begin{equation}\label{L4.24}
 C(\tilde x)=\limsup_{n\to\infty}\frac{a_n}{x_n},
\end{equation} (see Remark~\ref{R1}). Let $(n_j)_{j\in\mathbb N}$ be an infinite increasing sequence for which
\begin{equation}\label{L4.25}
 \lim_{j\to\infty}\frac{a_{n_j}}{x_{n_j}}=\limsup_{n\to\infty}\frac{a_n}{x_n}.
\end{equation} Define $r_j :=x_{n_j},$ $\tilde r:=(r_j)_{j\in\mathbb N}$ and
$t_j:=a_{n_j},$ $\tilde t:=(t_j)_{j\in\mathbb N}.$ It is clear that $\tilde r$ is a
normal scaling sequence. Relation \eqref{L4.23} and \eqref{L4.25} imply that $\tilde
t$ and $\tilde 0=(0,0,...,0,...)$ are mutually stable w.r.t. $\tilde r.$ Let $\tilde
E_{0, \tilde r}^{1}$ be a maximal (in $\tilde E^{1}$) self-stable family containing
$\tilde t$ and $\tilde 0.$ Using \eqref{L4.1}, \eqref{L4.24} and \eqref{L4.25}, we
obtain $$R^{*}(\mathbf{\Omega_{0}^{E^{1}}(n)})\ge\sup_{\tilde y\in\tilde E_{0,\tilde
r}^{1}}\tilde d_{\tilde r}(\tilde y, \tilde 0)\ge \tilde d_{\tilde r}(\tilde t, \tilde
0)=C(\tilde x).
$$ Hence \eqref{L4.20} holds that completes the proof of \eqref{L4.15}.
\end{proof}
The following theorem gives the necessary and sufficient conditions under which $\mathbf{\Omega_{p}^{X}(n)}$ is uniformly bounded.
\begin{theorem}\label{th3.10}
Let $(X,d,p)$ be a pointed metric space and let $E=S_{p}(X).$ The
family $\mathbf{\Omega_{p}^{X}(n)}$ is uniformly bounded if and only if
$E\in \textbf{CSP}.$ If $\mathbf{\Omega_{p}^{X}(n)}$ is
uniformly bounded and $p\in ac X$, then
\begin{equation}\label{L4.27}
R^{*}(\mathbf{\Omega_{p}^{X}(n)})=M(\tilde L)
\end{equation}
where $\tilde L$ is an universal element of $(\tilde
I_{E}^{d}, \preceq)$ and $M(\tilde L)$ is defined by \eqref{L13}.
\end{theorem}
\begin{proof} The theorem is trivial if $p$ is an isolated point of $X,$ so that we assume $p\in ac X.$ By Proposition~\ref{Pr4.4}, $\mathbf{\Omega_{p}^{X}(n)}$ is uniformly bounded
if and only if $\mathbf{\Omega_{0}^{E}(n)}$ is uniformly bounded. Since $C_E
=R^{*}(\mathbf{\Omega_{0}^{E}(n)})$ (see \eqref{L4.15}), $\mathbf{\Omega_{0}^{E}(n)}$ is
uniformly bounded if and only if $C_{E}<\infty.$ Using Proposition~\ref{P2} we obtain
that $\mathbf{\Omega_{0}^{E}(n)}$ is uniformly bounded if and only if $E\in \textbf{\emph{CSP}}.$

Let us prove that \eqref{L4.27} holds if $\mathbf{\Omega_{p}^{X}(n)}$ is uniformly
bounded and $p\in ac X$. In this case, as was proved above, $E\in \textbf{\emph{CSP}}.$
Consequently, by Theorem~\ref{ImpTh}, there is an universal element $\tilde L
\in\tilde I_{E}^{d}$ such that $M(\tilde L)<\infty.$
Lemma~\ref{Lem2.15} implies that
\begin{equation}\label{L4.29}
M(\tilde L)=C_E.
\end{equation}
By Proposition~\ref{Pr4.6}, we also have the equality
\begin{equation}\label{L4.30}
C_E=R^{*}(\mathbf{\Omega_{0}^{E}(n)}).
\end{equation}
Since $R^{*}(\mathbf{\Omega_{0}^{E}(n)})=R^{*}(\mathbf{\Omega_{p}^{X}(n)}),$ equalities \eqref{L4.29} and
\eqref{L4.30} imply \eqref{L4.27}.
\end{proof}

\begin{remark}
It is known (see \cite{dak}) that a bounded tangent space to $X$ at $p$ exists if and only if $S_{p}(X)=\{d(x,p): x\in X\}$
is strongly porous at 0. The necessary and sufficient conditions under which all pretangent spaces to $X$ at $p$ are bounded also formulated in terms of the local porosity of the set $S_{p}(X)$ (see \cite{bd1} for details).
\end{remark}

Theorem~\ref{th3.10} is an example of translation of some results related to completely strongly porous at 0 sets on the language of pretangent spaces. For more examples of such translation see Theorem~\ref{th4.4.5} and Proposition~\ref{Pr5.4} of the paper.





\begin{proposition}\label{Pr3.13}
Let $(X, d, p)$ be a pointed metric space and let $\mathbf{\Omega_{p}^{X}(n)}\ne\varnothing.$ If the family  $\mathbf{\Omega_{p}^{X}(n)}$ is uniformly bounded, then there is $\Omega_{p, \tilde r}^{X}\in\mathbf{\Omega_{p}^{X}(n)}$ such that the equality
\begin{equation}\label{eq3.28}
R^{*}\left(\mathbf{\Omega_{p}^{X}(n)}\right)=\rho^{*}\left(\Omega_{p, \tilde r}^{X}\right)
\end{equation}
holds.
\end{proposition}
\begin{proof}
Suppose that $\mathbf{\Omega_{p}^{X}(n)}$ is uniformly bounded. Write $E:=S_{p}(X).$ In the correspondence with Theorem~\ref{th3.10}, we have the equality
\begin{equation}\label{eq3.29}
R^{*}\left(\mathbf{\Omega_{p}^{X}(n)}\right)=M(\tilde L),
\end{equation}
where $\tilde L=\{(l_n, m_n)\}_{n\in\mathbb N}\in\tilde I_{E}^{d}$ is an universal element of $(\tilde I_{E}^{d}, \preceq)$ and
\begin{equation*}
M(\tilde L)=\limsup_{n\to\infty}\frac{l_n}{m_{n+1}}<\infty.
\end{equation*}
Let us consider a subsequence $\{(l_{n_k}, m_{n_k})\}_{k\in\mathbb N}$ of the sequence $\tilde L$ such that
\begin{equation}\label{eq3.30}
\lim_{k\to\infty}\frac{l_{n(k)}}{m_{n(k)+1}}=M(\tilde L).
\end{equation}
Since each $(l_n, m_n)$ is a connected component of $Ext E=Int(\mathbb R^{+}\setminus E),$ there is a sequence $(t_k)_{k\in\mathbb N}$ such that for every $k\in\mathbb N$, $t_k\in E$ and
\begin{equation}\label{eq3.31}
\lim_{k\to\infty}\frac{t_k}{l_{n(k)}}=1.
\end{equation}
Limit relations \eqref{eq3.30} and \eqref{eq3.31} imply
\begin{equation}\label{eq3.32}
\lim_{k\to\infty}\frac{t_k}{m_{n(k)+1}}=M(\tilde L).
\end{equation}
Passing, if it is necessary, to a subsequence again we may suppose that there is a decreasing sequence $(s_k)_{k\in\mathbb N}$ such that
\begin{equation}\label{eq3.33}
\lim_{k\to\infty}\frac{s_k}{m_{n(k)+1}}=1
\end{equation}
and $s_k\in E$ for every $k\in\mathbb N.$
Write $r_k:=m_{n(k)+1}$ for every $k\in\mathbb N.$ Let $(y_k)_{k\in\mathbb N}\in\tilde X$ be a sequence such that $t_k=d(p, y_k),$ $k\in\mathbb N.$ The triangle inequality implies that $$d(x_k, y_k)\le d(x_k, p)+d(p, y_k).$$ The last inequality, \eqref{eq3.32} and \eqref{eq3.33} imply that
\begin{equation*}
\limsup_{k\to\infty}\frac{d(x_k, y_k)}{r_k}\le\lim_{k\to\infty}\frac{d(x_k, p)}{r_k}+\lim_{k\to\infty}\frac{d(y_k, p)}{r_k}=1+M(\tilde L)<\infty.
\end{equation*}
Hence, the sequence $\left(\frac{d(x_k, y_k)}{r_k}\right)_{k\in\mathbb N}$ contains a convergent subsequence. Without loss of generality we may suppose that there is a finite limit
\begin{equation*}
\lim_{k\to\infty}\frac{d(x_k, y_k)}{r_k}.
\end{equation*}
Thus, the family $\{\tilde x, \tilde y, \tilde p\},$ where $\tilde x=(x_k)_{k\in\mathbb N},\, \tilde y=(y_k)_{k\in\mathbb N}\, \, \text{and}\,\, \tilde p=(p, p, ...)$ is self-stable. By the Zorn lemma, there is a maximal self-stable family $\tilde X_{p, \tilde r}$ such that
\begin{equation}\label{eq3.34}
\{\tilde x, \tilde y, \tilde p\}\subseteq\tilde X_{p, \tilde r}.
\end{equation}
Let $\Omega_{p, \tilde r}^{X}$ be the pretangent space obtained by metric identification of $\tilde X_{p, \tilde r}.$ Since $(s_k)_{k\in\mathbb N}$ is decreasing and equation \eqref{eq3.33} holds, the scaling sequence $(r_k)_{k\in\mathbb N}$ is normal, i.e. $\Omega_{p, \tilde r}^{X}\in\mathbf{\Omega_{p}^{X}(n)}.$ Let $\alpha=\pi(\tilde p),$ $\tilde p=(p, p, p, ...),$ and $\beta=\pi(\tilde y),$ where $\pi$ is the natural projection of $\tilde X_{p, \tilde r}$ on $\Omega_{p, \tilde r}^{X}.$ It follows from \eqref{eq3.32}, that
\begin{equation*}
\rho^{*}(\Omega_{p, \tilde r}^{X})=\sup_{\gamma\in\Omega_{p, \tilde r}^{X}}\rho(\alpha, \gamma)\ge\rho(\alpha, \beta)=M(\tilde L).
\end{equation*}
The inequality $\rho^{*}(\Omega_{p, \tilde r}^{X})\ge M(\tilde L)$ and \eqref{L4.20} imply
$
R^{*}(\mathbf{\Omega_{p}^{X}(n)})\le\rho^{*}(\Omega_{p, \tilde r}^{X}).
$
The converse inequality is trivial. Equality \eqref{eq3.28} follows.
\end{proof}

\begin{remark}
If $\mathbf{\Omega_{p}^{X}(n)}$ is uniformly bounded and nonempty, then it can be proved that there is $\Omega_{p, \tilde r}^{X}\in\mathbf{\Omega_{p}^{X}(n)}$ for which the equality
\begin{equation*}
\textrm{diam}\left(\Omega_{p, \tilde r}^{X}\right)=\sup_{\Omega\in\mathbf{\Omega_{p}^{X}(n)}} \textrm{diam}\left(\Omega\right)
\end{equation*}
holds.
\end{remark}

\section {Uniform boundedness and uniform discreteness}
\hspace*{\parindent} Let $I$ be an index set and let $\mathfrak F=\{(X_{i}, d_{i}, p_i): i \in I\}$ be a nonempty family
of pointed metric spaces.  We set
\begin{equation}\label{4.4.1}
\rho_{*}(X_i):=\begin{cases}
         \inf\{d_i(x,p_i): x\in X_{i}\setminus\{p_i\}\} & \mbox{if} $ $ X_i \ne \{p_i\}\\
         +\infty & \mbox{if}$ $ X_i = \{p_i\}\\
         \end{cases}
\end{equation}
for $i\in I$ and write
$$
R_{*}(\mathfrak F):=\mathop{\inf}\limits_{i\in I}\rho_{*}(X_i).
$$
If $R_{*}(\mathfrak F)>0,$ then we say that $\mathfrak F$ is \emph{uniformly discrete} (w.r.t. the points $p_i$)

As in Proposition~\ref{Pr4.1} it is easy to show that the family $\mathbf {\Omega_{p}^{X}}$ of all pretangent spaces is uniformly discrete if and only if $p$ is an isolated point of the metric space $X.$ Thus, it make sense to consider $\mathbf {\Omega_{p}^{X}(n)}.$

\begin{theorem}\label{th4.4.1}
Let $(X,d,p)$ be a pointed metric space such that $\mathbf {\Omega_{p}^{X}(n)}\ne\varnothing.$ Then $\mathbf {\Omega_{p}^{X}(n)}$ is uniformly discrete if and only if it is uniformly bounded. If $\mathbf {\Omega_{p}^{X}(n)}$ is uniformly bounded, then the equality
\begin{equation}\label{4.4.3}
R_{*}(\mathbf {\Omega_{p}^{X}(n)})=\frac{1}{R^{*}(\mathbf {\Omega_{p}^{X}(n)})}
\end{equation} holds.
\end{theorem}

\begin{remark}\label{r4.4.2}
The condition $\mathbf {\Omega_{p}^{X}(n)}\ne\varnothing$ implies $R^{*}(\mathbf {\Omega_{p}^{X}(n)})>0.$
\end{remark}

We will prove Theorem~\ref{th4.4.1} in some more general setting.

Let $(X,d,p)$ be a pointed metric space and let $t>0.$ Write $t X$ for the pointed metric space $(X, t d, p)$, i.e. $t X$ is the pointed metric space with the same underluing set $X$ and the marked point $p,$ but equipped with the new metric $t d$ instead of $d$.

\begin{definition}\label{D4.4.2}
Let $\mathfrak F$ be a nonempty family of pointed metric spaces. $\mathfrak F$ is weakly self-similar if for every $(Y,d,p)\in\mathfrak F$ and every nonzero $t \in S_{p}(Y)$ the space $\frac{1}{t}Y$ belongs to $\mathfrak F.$
\end{definition}

\begin{theorem}\label{th4.4.3}
Let $\mathfrak F=\{(X_i, d_i, p_i): i\in I\}$ be a weakly self-similar family of pointed metric spaces. Suppose that the sphere
\begin{equation*}
S_{i}=\{x\in X_i: d_{i}(x, p_i)=1\}
\end{equation*}
is nonvoid for every $i\in I.$ Then $\mathfrak F$ is uniformly bounded if and only if $\mathfrak F$ is uniformly discrete w.r.t. the marked points $p_i, i\in I.$ If $\mathfrak F$ is uniformly bounded, then the equality
\begin{equation}\label{4.4.4}
R_{*}(\mathfrak F)=\frac{1}{R^{*}(\mathfrak F)}
\end{equation}
holds.
\end{theorem}

\begin{proof}
Assume that $\mathfrak F$ is uniformly bounded but not uniformly discrete. Then there is a sequence $(x_{i_k})_{k\in\mathbb N}$ such that

\begin{equation*}\label{4.4.5}
x_{i_k}\in X_{i_k}, x_{i_k}\ne p_{i_k} \,\, \mbox{and} \, \, \lim_{k\to\infty}d_{i_k}(x_{i_k}, p_{i_k})=0
\end{equation*}
Since all $S_{i}$ are nonvoid, we can find a sequence $(y_{i_k})_{k\in\mathbb N}$ for which $y_{i_k}\in X_{i_k}$ and
$
d_{i_k}(y_{i_k}, p_{i_k})= 1
$
for every $k\in\mathbb N.$  Define $t_k$ to be $d_{i_k}(x_{i_k}, p_{i_k}), k\in\mathbb N.$ Since $\mathfrak F$ is weakly self-similar,
$
t_{k}^{-1}X_{i_k}\in\mathfrak F
$
holds for every $k\in\mathbb N.$ Now we obtain

\begin{equation*}
R^{*}(\mathfrak F)\ge\limsup_{k\to\infty}t_{k}^{-1}d_{i_k}(y_{i_k}, p_{i_k})=\limsup_{k\to\infty}t_{k}^{-1}=\infty.
\end{equation*}
Hence $\mathfrak F$ is not uniformly bounded, contrary to the assumption. Therefore if $\mathfrak F$ is uniformly bounded, then $\mathfrak F$ is uniformly discrete. Similarly we can prove that the uniform discreteness of $\mathfrak F$ implies the uniform boundedness of this family.

Suppose now that
$
R^{*}(\mathfrak F)<\infty.
$
Let us prove equality \eqref{4.4.4}. Define a quantity $Q(\mathfrak F)$ by the rule

\begin{equation*}
Q(\mathfrak F)=\sup_{i\in I}\frac{\rho^{*}(X_i)}{\rho_{*}(X_i)}
\end{equation*}
where $\rho_{*}(X_i)$ is defined by \eqref{4.4.1} and $\rho^{*}(X_i)$ by \eqref{L4.1}.The first part of the theorem implies that $\mathfrak F$ is uniformly discrete. Hence $R_{*}(\mathfrak F)>0,$ that implies $$\rho_{*}(X_i)>0, i\in I.$$ Moreover, the inequality $R^{*}(\mathfrak F)<\infty$ gives us the condition $\rho_{*}(X_i)<\infty.$ Thus $Q(\mathfrak F)$ is correctly defined.
We claim that the equality
$
Q(\mathfrak F)=R^{*}(\mathfrak F)
$
holds. Indeed let $(i_k)_{k\in\mathbb N}$ be a sequence of indexes $i_k\in I$ such that

\begin{equation}\label{4.4.10}
\lim_{k\to\infty}\rho_{*}(X_{i_k})=R^{*}(\mathfrak F).
\end{equation}
Since all $S_{i}$ are nonvoid, we have $\rho_{*}(X_{i_k})\le 1$ for every $X_{i_k}.$ Consequently
\begin{equation}\label{4.4.11}
Q(\mathfrak F)\ge\limsup_{k\to\infty}\frac{\rho^{*}(X_{i_k})}{\rho_{*}(X_{i_k})}\ge\limsup_{k\to\infty}\rho^{*}(X_{i_k})=\lim_{k\to\infty}\rho^{*}(X_{i_k})\ge R^{*}(\mathfrak F).
\end{equation}
Let us consider a sequence $(i_m)_{m\in\mathbb N}, i_m\in I,$ for which
\begin{equation}\label{4.4.12}
Q(\mathfrak F)=\lim_{m\to\infty}\frac{\rho^{*}(X_{i_m})}{\rho_{*}(X_{i_m})}.
\end{equation}
The quantity $\frac{\rho^{*}(X_{i})}{\rho_{*}(X_{i})}$ is invariant w.r.t. the passage from $X_i$ to $\frac{1}{t}X_{i}$ if $t\in S_{p_i}(X_i).$ Consequently using the uniform discreteness of $\mathfrak F$ and the inequality $\rho_{*}(X_i)\le 1$ (which follows from the condition $S_{i}\ne\varnothing, \, i\in I$), we may assume that

\begin{equation}\label{4.4.13}
\lim_{m\to\infty}\rho_{*}(X_{i_m})=1.
\end{equation}
Limit relations \eqref{4.4.12} and \eqref{4.4.13} imply

\begin{equation*}
Q(\mathfrak F)=\lim_{m\to\infty}\rho_{*}(X_{i_m})\le R^{*}(\mathfrak F).
\end{equation*}
The last inequality and \eqref{4.4.11} give us the equality
$R^{*}(\mathfrak F)=Q(\mathfrak F).$
Reasoning similarly we obtain the equality
$
Q(\mathfrak F)=\frac{1}{R_{*}(\mathfrak F)}.
$
Equality \eqref{4.4.4} follows.
\end{proof}

Let us define a subset $^{\textbf{1}}\mathbf{\Omega_{p}^{X}}$ of the set $\mathbf{\Omega_{p}^{X}}$ of all pretangent spaces to $X$ at $p$ by the rule:
\begin{equation*}\label{4.4.16}
\left(\left(\Omega_{p,\tilde r}^{X}, \rho, \alpha\right)\in \,^{\textbf{1}}\mathbf{\Omega_{p}^{X}}\right)\Leftrightarrow
\left(\{\delta\in\Omega_{p,\tilde r}^{X}: \rho(\alpha,\delta)=1\}\ne\varnothing\right)
\end{equation*}
where $\alpha=\pi(\tilde p)$ is the marked point of the pretangent space $\Omega_{p,\tilde r}^{X}$ and $\rho$ is the metric on $\Omega_{p,\tilde r}^{X}.$
To apply Theorem~\ref{th4.4.3} to the proof of Theorem~\ref{th4.4.1} we need the following
\begin{lemma}\label{l4.4.4}
Let $(X, d, p)$ be a pointed metric space.
If \, $\Omega_{p,\tilde r}^{X}\in\mathbf{\Omega_{p}^{X}(n)},$ then there is $^{1}\Omega_{p,\tilde\mu}^{X}\in\,^{\textbf{\emph{1}}}\mathbf{\Omega_{p}^{X}}$ such that
\begin{equation}\label{4.4.18}
\rho_{*}(^{1}\Omega_{p,\tilde\mu}^{X})\le\rho_{*}(\Omega_{p,\tilde r}^{X})\le\rho^{*}(\Omega_{p,\tilde r}^{X})\le\rho^{*}(^{1}\Omega_{p,\tilde \mu}^{X}).
\end{equation}
Conversely, if $\, ^{1}\Omega_{p, \tilde r}^{X}\in\, ^{\textbf{\emph{1}}}\mathbf{\Omega_{p}^{X}},$ then there is $\Omega_{p, \tilde \mu}^{X}\in\mathbf{\Omega_{p}^{X}(n)}$ such that
\begin{equation}\label{4.4.18*}
\rho_{*}(\Omega_{p,\tilde\mu}^{X})\le\rho_{*}(^{1}\Omega_{p,\tilde r}^{X})\le\rho^{*}(^{1}\Omega_{p,\tilde r}^{X})\le\rho^{*}(\Omega_{p,\tilde \mu}^{X}).
\end{equation}
\end{lemma}

\begin{proof}
Let $\Omega_{p, \tilde r}^{X}\in\mathbf{\Omega_{p}^{X}(n)}$ and let
$\tilde X_{p,\tilde r}$ be the maximal self-stable family which metric identification coincides with $\Omega_{p,\tilde r}^{X}.$ We can find some sequences $\tilde a^{i}=(a_n^{i})_{n\in\mathbb N}$ and $\tilde b^{i}=(b_n^{i})_{n\in\mathbb N}, i\in\mathbb N$ such that

\begin{equation}\label{4.4.20}
\lim_{i\to\infty}\rho(\pi(\tilde b^{i}), \alpha)=\rho_{*}(\Omega_{p,\tilde r}^{X})\quad\mbox{and}\quad \lim_{i\to\infty}\rho(\pi(\tilde a^{i}), \alpha)=\rho^{*}(\Omega_{p,\tilde r}^{X}).
\end{equation}
Since $\Omega_{p,\tilde r}^{X}\in\mathbf{\Omega_{p}^{X}(n)},$ the scaling sequence $\tilde r$
is normal. Consequently there is $\tilde c=(c_n)_{n\in\mathbb N}\in\tilde X$ such that $$\lim_{n\to\infty}\frac{d(c_n,p)}{r_n}=1.$$ Let us define a countable family $\mathfrak B\subseteq\tilde X$ as $$\mathfrak B=\{\tilde b^{i}: i\in\mathbb N\}\cup\{\tilde a^{i}: i\in\mathbb N\}\cup\{\tilde c\}.$$ The family $\mathfrak B$ satisfies the condition of Lemma~\ref{Lem1.6}. Consequently there is an infinite subsequence $\tilde r'=(r_{n_k})_{k\in\mathbb N}$ the scaling sequence $\tilde r$ for which the family $$\mathfrak B'=\{(b_{n_k}^{i})_{k\in\mathbb N}: i\in\mathbb N\}\cup\{(a_{n_k}^{i})_{k\in\mathbb N}: i\in\mathbb N\}\cup\{(c_{n_k})_{k\in\mathbb N}\}$$ is self-stable w.r.t. $\tilde r'$. Completing $\mathfrak B'$ to a maximal self-stable family and passing to the metric identification, we obtain the desired pretangent space $^{1}\Omega_{p,\tilde\mu}^{X}$ with $\tilde\mu=\tilde r'.$ The second statement of the lemma follows directly from statement $(i_2)$ of Proposition~\ref{prop}.
\end{proof}
\noindent{\bf Proof of Theorem~\ref{th4.4.1}.} Lemma~\ref{l4.4.4} implies the equalities
\begin{equation}\label{4.4.21}
R^{*}(^{\textbf{1}}\mathbf{\Omega_{p}^{X}})= R^{*}(\mathbf{\Omega_{p}^{X}(n)})\quad\mbox{and}\quad
R_{*}(\mathbf{\Omega_{p}^{X}(n)})=R_{*}(^{\textbf{1}}\mathbf{\Omega_{p}^{X}}).
\end{equation}
Now Theorem~\ref{th4.4.1} follows from Theorem~\ref{th4.4.3} and \eqref{4.4.21}.$\qquad\qquad\quad\quad\quad\quad\quad\quad\blacksquare$

\medskip

The sum of theorems~\ref{th4.4.1} and \ref{th3.10} yields the following result.

\begin{theorem}\label{th4.4.5}
Let $(X,d,p)$ be a metric space with a marked point $p\in ac X.$ Then the following three conditions are equivalent.
\item[\rm(i)]\textit{$\mathbf{\Omega_{p}^{X}(n)}$ is uniformly bounded.}
\item[\rm(ii)]\textit{$\mathbf{\Omega_{p}^{X}(n)}$ is uniformly discrete.}
\item[\rm(iii)]\textit{$S_{p}(X)\in \textbf{CSP}.$ }

Moreover if $\mathbf{\Omega_{p}^{X}(n)}$ is uniformly bounded, then
\begin{equation*}
R^{*}(\mathbf{\Omega_{p}^{X}(n)})=M(\tilde L)\quad\mbox{and}\quad
R_{*}(\mathbf{\Omega_{p}^{X}(n)})=\frac{1}{M(\tilde L)}
\end{equation*} where the quantity $M(\tilde L)$ was defined by \eqref{L13}.
\end{theorem}

The following proposition is an analog of Proposition~\ref{Pr3.13}.
\begin{proposition}\label{Pr4.7}
Let $(X, d, p)$ be a pointed metric space and let $\mathbf{\Omega_{p}^{X}(n)}\ne\varnothing.$
If the family $\mathbf{\Omega_{p}^{X}(n)}$ is uniformly descrete, then there is $\Omega_{p, \tilde r}^{X}\in\mathbf{\Omega_{p}^{X}(n)}$ such that the equality
\begin{equation*}
R_{*}(\mathbf{\Omega_{p}^{X}(n)})=\rho_{*}(\Omega_{p, \tilde r}^{X})
\end{equation*} holds.
\end{proposition}
A proof is completely similar to the proof of Proposition~\ref{Pr3.13}.

\section {Metric space valued derivatives on uniformly bounded pretangent spaces}
\hspace*{\parindent}
Let $(X_i, d_i, p_i)$ be pointed metric spaces and let $\tilde r_i=\{r_{n}^{(i)}\}_{n\in\mathbb N}$ be scaling sequences, $i=1, 2.$  For a function $f: X_1\to X_2$ define the mapping $\tilde f:\tilde X_1\to\tilde X_2$ as
\begin{equation*}
\tilde f(\tilde x)=(f(x_1), f(x_2), ..., f(x_n), f(x_{n+1}), ...)
\end{equation*}
if $\tilde x=(x_i)_{i\in \mathbb N}\in\tilde X_1.$ Let us consider two maximal self-stable families $\tilde X_{p_1, \tilde r_1}^{1}\subseteq\tilde X_1$ and $\tilde X_{p_2, \tilde r_2}^{2}\subseteq\tilde X_2.$
\begin{definition}\label{diff}
A function $f: X_1\to X_2$ is differentiable w.r.t. the pair $\left(\tilde X_{p_1, \tilde r_1}^{1}, \tilde X_{p_2, \tilde r_2}^{2}\right)$ if the following conditions are satisfied:
\item[\rm(i)]\textit{$\tilde f(\tilde x)\in\tilde X_{p_2, \tilde r_2}^{2}$ for every $\tilde x\in\tilde X_{p_1, \tilde r_1}^{1}$;}
\item[\rm(ii)]\textit{The implication $(\tilde d_{\tilde r_1}(\tilde x, \tilde y)=0)\Rightarrow (\tilde d_{\tilde r_2}(\tilde f(\tilde x), \tilde f(\tilde y))=0)$ is true for all $\tilde x, \tilde y\in \tilde X_{p_1, \tilde r_1}^{1}$, where
    \begin{equation*}
    \tilde d_{\tilde r_1}(\tilde x, \tilde y)=\lim_{n\to\infty}\frac{d_{1}(x_n, y_n)}{r_{n}^{(1)}},
    \end{equation*}
    \begin{equation*}
    \tilde d_{\tilde r_2}(\tilde f(\tilde x), \tilde f(\tilde y))=\lim_{n\to\infty}\frac{d_{2}(f(x_n), f(y_n))}{r_{n}^{(2)}}.
    \end{equation*}}
\end{definition}

\begin{remark}
Note that condition (i) of Definition~\ref{diff} implies the equality $f(p_1)=p_2.$
\end{remark}

Let $\pi_{X_i}:\tilde X_{p_i, \tilde r_i}^{i}\to\Omega_{p_i, \tilde r_i}$, $i=1, 2,$ be natural projections. (See diagram~\eqref{eq1.5}).

\begin{definition}(\cite{dov})
A function $D^{*}f:\Omega_{p_1, \tilde r_1}\to\Omega_{p_2, \tilde r_2}$ is a metric space valued derivative of $f: X_1\to X_2$ at the point $p_1\in X_1$ if $f$ is differentiable w.r.t. $\left(\tilde X_{p_1, \tilde r_1}^{1}, \tilde X_{p_2, \tilde r_2}^{2}\right)$ and the diagram
\begin{equation}
\begin{array}{ccc}
\tilde X_{p_1, \tilde r_{1}}^{1} & \xrightarrow{\ \ \quad\tilde f\quad \ \ } &
\tilde X_{p_2, \tilde r_{2}}^{2} \\
\!\! \!\! \!\! \!\! \! \pi_{X_1}\Bigg\downarrow &  & \! \!\Bigg\downarrow \pi_{X_2}
\\
\Omega_{p_1, \tilde r_1} & \xrightarrow{\ \ \quad D^{*}f \quad\ \ } & \Omega_{p_2, \tilde
r_{2}}
\end{array}
\end{equation}
is commutative.
\end{definition}

The existence of $D^{*}f$ for differentiable w.r.t. $(\tilde X_{p_1, \tilde r_1}^{1}, \tilde X_{p_2, \tilde r_2}^{2})$ functions $f$ directly follows from Definition~\ref{diff}.

Let $(X, d)$ be a metric space and $Y$ be a set. Recall that a function $f: X\to Y$ is a locally constant at a point $p\in X$ if there is a neighborhood $U$ of $p$ such that $f$ is constant on $U.$ Write
\begin{equation}\label{eq5.2}
c_p(f):=
\begin{cases}
\inf\{d(x, p): f(x)\ne f(p)\}\,\text{if}\, f\, \text{is not constant on}\, X\\
\infty \,\qquad\qquad\qquad\qquad\qquad\text{if}\, f\, \text{is constant on}\, X.\\
\end{cases}
\end{equation}
Then $f$ is locally constant at $p\in X$ if and only if $c_{p}(f)>0.$

In the following proposition the quantity $M(\tilde L)$ was defined by \eqref{L13}.
\begin{proposition}\label{Pr5.4}
Let $(X, d, p)$ be a pointed metric space with $\mathbf{\Omega_{p}^{X}(n)}\ne\varnothing.$ Then the following statements are equivalent.
\item[\rm(i)]\textit{$S_{p}(X)\in$ \textbf{CSP}.}
\item[\rm(ii)]\textit{There is a constant $c_{X}\in (0, \infty)$ such that if $Y$ is a pointed metric space with a marked point $a$, $\tilde X_{p, \tilde r}\subseteq\tilde X$ and $\tilde Y_{a, \tilde t}\subseteq\tilde Y$ are maximal self-stable families, and $\Omega_{p, \tilde r}^{X}=\pi_{X}(\tilde X_{p, \tilde r})\in \mathbf{\Omega_{p}^{X}(n)}$ and $\Omega_{a, \tilde t}^{Y}=\pi_{Y}(\tilde Y_{a, \tilde t})\in \mathbf{\Omega_{a}^{Y}}$, and $f: X\to Y$ is differentiable w.r.t. $(\tilde X_{p, \tilde r}, \tilde Y_{a, \tilde t})$, then the derivative $D^{*}f: \Omega_{p, \tilde r}^{X}\to \Omega_{a, \tilde t}^{Y}$ satisfies the inequality
    \begin{equation}\label{eq5.3}
    c_{\alpha}(D^{*}f)\ge c_{X},
    \end{equation} where $\alpha=\pi_{X}(\tilde p).$
Moreover, if condition \emph{(i)} holds, then we can take $c_{X}=\frac{1}{M(\tilde L)}$ and this bound is sharp for inequality \eqref{eq5.3}}.
\end{proposition}

\begin{proof}
Let (i)  holds. Then, by Theorem~\ref{th4.4.5}, $\mathbf{\Omega_{p}^{X}(n)}$ is uniformly discrete and
\begin{equation*}
R_{*}(\mathbf{\Omega_{p}^{X}(n)})=\frac{1}{M(\tilde L)}>0.
\end{equation*}
Hence, for $r\in\left(0, \frac{1}{M(\tilde L)}\right),$ the open ball $B_{r}(\alpha)$ of radius $r$ centered at $\alpha=\pi_{X}(\tilde p)$ contains the point $\alpha$ only. Consequently, for every function $g$ defined on $\Omega_{p, \tilde r}^{X}\in\mathbf{\Omega_{p}^{X}(n)}$ we have $$c_{\alpha}(g)\ge\frac{1}{M(\tilde L)}.$$
The implication (i)$\Rightarrow$(ii) follows. It is also shown that inequality \eqref{eq5.3} holds with $c_{X}=\frac{1}{M(\tilde L)}$ for every $D^{*}(f):\Omega_{p, \tilde r}^{X}\to\Omega_{a, \tilde t}^{Y}$ if statement (i) is valid. Suppose now that (ii) holds. By Theorem~\ref{th4.4.5}, statement (i) holds if and only if the family $\mathbf{\Omega_{p}^{X}(n)}$ is uniformly discrete. Let $\tilde X_{p, \tilde r}$ be a maximal self-stable family and let $\Omega_{p, \tilde r}^{X}=\pi_{X}(\tilde X_{p, \tilde r})\in\mathbf{\Omega_{p}^{X}(n)}.$ Let us consider the identical mappings
\begin{equation*}
id_{X}: X\rightarrow X\quad\text{and}\quad id_{\Omega}:\Omega_{p, \tilde r}^{X}\rightarrow\Omega_{p, \tilde r}^{X}.
\end{equation*}
It is easy to show that $id_{X}$ is differentiable w.r.t. $(\tilde X_{p, \tilde r}, \tilde X_{p, \tilde r})$ with the metric space value derivative $D^{*}(id_X)=id_{\Omega}.$ By condition (ii), we have
$$c_{\alpha}(id_{\Omega})\ge c_{X}.$$
Moreover, from the definitions of $c_{\alpha}$ and $\rho_{*}$ it follows that
\begin{equation}\label{eq5.4}
c_{\alpha}(id_{\Omega})=\rho_{*}(\Omega_{p, \tilde r}^{X})\quad\text{with}\quad\alpha=\pi_{X}(\tilde p).
\end{equation}
Consequently, the inequality
\begin{equation}\label{eq5.5}
R_{*}(\mathbf{\Omega_{p}^{X}(n)})\ge c_{X}
\end{equation}
holds. Since $c_{X}>0,$ inequality \eqref{eq5.5} implies that $\mathbf{\Omega_{p}^{X}(n)}$ is uniformly discrete.
To complete the proof, it suffices to show that $c_{X}=\frac{1}{M(\tilde L)}$ is the best possible bound in \eqref{eq5.3}. Indeed, by Proposition~\ref{Pr4.7}, there is $^{1}\Omega_{p, \tilde r}^{X}\in\mathbf{\Omega_{p}^{X}(n)}$ such that
\begin{equation}\label{eq5.6}
\rho_{*}(^{1}\Omega_{p, \tilde r}^{X})=R_{*}(\mathbf{\Omega_{p}^{X}(n)}).
\end{equation}
From Theorem~\ref{th4.4.5} it follows that
\begin{equation}\label{eq5.7}
R_{*}(\mathbf{\Omega_{p}^{X}(n)})=\frac{1}{M(\tilde L)}.
\end{equation}
Now using \eqref{eq5.4}, \eqref{eq5.6} and \eqref{eq5.7}, we obtain
$$c_{\alpha}(id_{\Omega})=\frac{1}{M(\tilde L)},$$
where $id_{\Omega}$ is the identical mapping of $^{1}\Omega_{p, \tilde r}^{X}.$ The bound $c_X=\frac{1}{M(\tilde L)}$ is sharp.
\end{proof}
\begin{remark}
Recently, Dmytro Dordovskyi proved that pretangent spaces at regular points of $k$-dimensional parametric manifolds in $\mathbb R^{n}$ are isometric to $\mathbb R^{k}$ \cite{dd}. Using this result, we can show that some metric space valued derivatives of mappings between such manifolds can be identified with linear mappings between Euclidean spaces. Every locally constant linear mapping is identically zero, so Proposition~\ref{Pr5.4} gives us conditions under which every metrically differentiable function undergoes a ``local vanishing'' of metric derivative.
\end{remark}

\end{document}